\documentclass[pdflatex,sn-mathphys-num]{sn-jnl}%
\usepackage[scr=rsfso]{mathalfa}%
\let\orcidlogo\relax%
\usepackage{orcidlink}%
\usepackage[utf8]{inputenc}
\usepackage[T1]{fontenc}
\usepackage{lmodern}
\usepackage{graphicx}%
\usepackage{multirow}%
\usepackage{amsmath,amssymb,amsfonts}%
\usepackage{amsthm}%

\usepackage[title]{appendix}%
\usepackage{xcolor}%
\usepackage{textcomp}%
\usepackage{mathtools}
\allowdisplaybreaks
\usepackage{enumitem}
\usepackage{tikz-cd}
\tikzcdset{arrow style=tikz, diagrams={>=stealth}}
\newcommand{\PP}{\mathbb{P}}
\newcommand{\QQ}{\mathbb{Q}}
\newcommand{\ZZ}{\mathbb{Z}}
\newcommand{\OO}{\mathcal{O}}
\DeclareMathOperator{\Gal}{Gal}
\DeclareMathOperator{\Res}{Res}
\DeclareMathOperator{\Disc}{Disc}

\DeclareMathOperator{\Aut}{Aut}
\DeclareMathOperator{\Hom}{Hom}
\DeclareMathOperator{\Crit}{Crit}
\DeclareMathOperator{\PC}{PC}
\DeclareMathOperator{\Br}{Br}

\theoremstyle{thmstyleone}
\newtheorem{theorem}{Theorem}[section]
\newtheorem{proposition}[theorem]{Proposition}
\newtheorem{lemma}[theorem]{Lemma}
\newtheorem{corollary}[theorem]{Corollary}

\newtheorem{question}[theorem]{Question}
\theoremstyle{thmstyletwo}
\newtheorem{remark}[theorem]{Remark}
\newtheorem{example}[theorem]{Example}
\theoremstyle{thmstylethree}
\newtheorem{definition}[theorem]{Definition}

\begin{document}

\title[A Dynamical Néron--Ogg--Shafarevich Criterion]{A Dynamical Néron--Ogg--Shafarevich Criterion via Orbital Arboreal Representations}

\author*[1]{\fnm{J. Rogelio} \sur{Pérez-Buendía}\,\orcidlink{0000-0002-7739-4779}}\email{rogelio.perez@cimat.mx}
\affil*[1]{\orgname{SECIHTI--CIMAT, Unidad Mérida}, \orgaddress{\city{Mérida}, \state{Yucatán}, \country{México}}}

\abstract{Let $K$ be a non-archimedean local field and $\varphi:\PP^1\!\to\!\PP^1$ a rational endomorphism of degree $d\ge2$ defined over $K$. Working with a normalized integral homogeneous lift, we show, in the tame case ($p\nmid d$), that strict good reduction is equivalent to the existence of a nonempty Zariski open subset $U_k\subset\PP^1_k\!\setminus\!\PC(\tilde\varphi)$ over which the canonical residual morphism is finite étale of degree $d$. The criterion separates two complementary local invariants of the model: the resultant $\Res(F,G)$ controls residual degree drop, while the fiber discriminants $\Disc(F_{n,x})$ control étaleness of the corresponding residual fibers once full residual degree is ensured. Consequently, for every finite $x\in\OO_K$ whose reduction lies in $U_k$, the projective fibre is finite \'etale of degree~$d^n$; when additionally $\bar x\neq\tilde\varphi^n(\infty)$, the affine fiber polynomial has unit leading coefficient and unit discriminant. In either case, all backward preimage extensions $K(X_n(x))/K$ are unramified for all $n\ge1$. We also introduce the \emph{orbital preimage tree} $T_{O^+(x)}=\varinjlim_n X_\infty(\varphi^n(x))$, the colimit in $G_K$-sets of the inverse preimage trees along the branch inclusions determined by the forward orbit, and the associated \emph{orbital arboreal Galois image} $\mathcal{G}_{O^+(x)}=\operatorname{Im}(G_K\to\Aut(T_{O^+(x)}))$. We isolate the forward-invariant safe locus $U_k^{\mathrm{safe}}=\bigcap_{m\ge 0}\tilde\varphi^{-m}(U_k)\subseteq U_k$ and the corresponding residually admissible integral points, on which strict good reduction is captured by the bijectivity of the orbital reduction map $X_n(x_m)\to\widetilde X_n(\bar x_m)$. This provides a canonical orbit-invariant arboreal framework that integrates with the local criterion and connects with arboreal Galois representations as developed by Boston–Jones, Jones, and others. The criterion is computable, propagates from level $1$ to all backward preimage levels, and yields both pointwise and orbit-level reformulations of arboreal Galois questions. We provide complete proofs and explicit examples over $\QQ_p$.}


\maketitle

\noindent
\textbf{Published in:} Res.\ Number Theory \textbf{12} (2026), article~60.
\href{https://doi.org/10.1007/s40993-026-00748-9}{doi:10.1007/s40993-026-00748-9}

\bigskip

\section{Introduction}\label{sec:intro}

Good reduction precludes the emergence of new ramification in the Galois tower, thereby facilitating the analysis of rational dynamical systems over local fields. This phenomenon establishes connections among non-archimedean analysis, arithmetic geometry, and Galois theory. A central principle, analogous to the Néron-Ogg-Shafarevich criterion for abelian varieties, asserts that good reduction is characterized by the absence of ramification in the natural towers associated with the system (see, e.g., \cite{Benedetto2019, BenedettoActa2014, BenedettoIMRN2015, SilvermanADS}). For results on degree bounds for attaining potentially good reduction after a finite extension, see \cite{BenedettoIMRN2015}. In addition to abelian varieties, $p$-adic NOS-type criteria have been established for curves via the $p$-adic unipotent fundamental group \cite{AndreattaIovitaKim2015}, and for $K3$ surfaces using monodromy and crystalline techniques \cite{HernandezMadaPadova2021, LiedtkeMatsumoto2018, PerezBuendiaAMQ2019}.

A rigorous Galois criterion for rational maps, however, must avoid a common pitfall: quantifying over \emph{all} $x \in \PP^1(\overline K)$ fails, as a point $x$ may itself generate a ramified extension $K(x)/K$. This would make the entire tower $K(X_n(x))/K$ ramified, regardless of the dynamics of $\varphi$. A precise criterion must therefore restrict the test points $x$ to a suitable unramified integral locus where the dynamics truly determine the ramification.

To formulate such a criterion in an orbit-invariant arboreal form, we introduce the \emph{orbital preimage tree}
\[
T_{O^+(x)}\ :=\ \varinjlim_n X_\infty(\varphi^n(x)),
\]
obtained by gluing the backward preimage trees along the forward orbit $O^+(x)$ via natural branch inclusions, together with the associated \emph{orbital arboreal Galois image} $\mathcal{G}_{O^+(x)}:=\operatorname{Im}(G_K\to\Aut(T_{O^+(x)}))$. This is a canonical, orbit-invariant arboreal Galois object: $T_{O^+(x)}$ is characterized by a universal property as a colimit in the category of $G_K$-sets (Proposition~\ref{prop:universal-T}). Theorem~\ref{thm:main} packages the classical strict-good-reduction criterion with the residual étale-locus condition needed to control ramification in the backward preimage towers, working over the locus $U_k = \PP^1_k \setminus \PC(\widetilde\varphi)$. Consequently, over this locus the backward preimage extensions $K(X_n(x))/K$ are unramified for every finite $x\in\OO_K$ with $\bar x\in U_k$, and inertia $I_v$ acts trivially on $X_\infty(x)$. At the orbit level, strict good reduction is reflected not merely by inertia-triviality but by the bijectivity of the orbital reduction maps (\S\ref{sec:orbital-open-image}).

The local theory of arithmetic dynamics now sits at the meeting point of four mature lines of work: nonarchimedean reduction theory \cite{Ben01}, branch-locus and critical-good-reduction criteria \cite{Can11, Can15}, arboreal Galois representations \cite{BostonJones, JonesBesancon2013}, and the ramification of backward towers \cite{Ait04, Cul12, Bri15}. Within this picture, postcritical behavior emerges as the principal arithmetic obstruction to ramification: when the postcritical orbit is finite, preimage towers are finitely ramified \cite{Ait04, Bri15}, while delicate phenomena can persist locally even in seemingly tame regimes \cite{And16}. The dynamical Néron--Ogg--Shafarevich question sits precisely at the intersection of these threads, and recent work has approached it from several angles. Benedetto's criteria \cite{BenedettoActa2014, BenedettoIMRN2015} address \emph{potential} good reduction---whether a rational map acquires good reduction after a finite extension---while the preprint of Adams \cite{adams2022} adopts an anabelian standpoint, working with Galois actions on fundamental-group objects attached to complements of precritical loci, and characterizes \emph{good critical reduction}. This adds to residual degree preservation a non-collision requirement on the critical and branch data: every map with good critical reduction has strict good reduction, but not conversely (\cite{Can11} records the precise relationship under separability).

The present paper proves a direct characterization of \emph{strict} good reduction over the ground field itself, in the tame case $p\nmid d$, working in the arboreal setting with pointwise preimage towers and their orbital colimit rather than with anabelian fundamental-group objects. Theorem~\ref{thm:main} provides an \emph{orbital refinement} of the classical local criterion \cite[Theorem~2.15]{SilvermanADS}, specialized to the residual étale locus $U_k$. The local criterion separates two complementary invariants of the residual model: the resultant $\Res(F,G)$ controls residual degree drop, equivalently whether the canonical residual morphism $\widetilde\varphi$ has full degree $d$, while, once full residual degree is ensured, the fiber discriminants $\Disc(F_{n,x})$ control étaleness of the corresponding residual fibers (Remark~\ref{rem:two-invariants}). The unit leading coefficient and unit discriminant conditions on $F_{1,x}$ then propagate to all backward preimage levels by persistence of the unit resultant under iteration. The orbital framework then packages the ramification data attached to the entire forward orbit into a single canonical Galois object, the orbital arboreal image $\mathcal{G}_{O^+(x)}$, while remaining within the arboreal Galois setting. Restricting to the forward-invariant safe locus $U_k^{\mathrm{safe}}\subseteq U_k$, strict good reduction admits a clean orbit-level reformulation in terms of bijective specialization of the orbital reduction map (Section~\ref{sec:orbital-open-image}); pure inertia-triviality on the orbital tree is a necessary but not sufficient condition, as a one-line example over $\QQ_5$ illustrates.

\section{Preliminaries}\label{sec:prelim}
\subsection{Good reduction of rational maps}\label{subsec:goodred}
Let $K$ be a non-archimedean local field, with $\OO=\OO_K$ its ring of integers, $\mathfrak m$ its maximal ideal, and $k=\OO/\mathfrak m$ its (finite, perfect) residue field. We write
\[
\varphi([X:Y])=[F(X,Y):G(X,Y)],\quad F,G\in K[X,Y]\text{ homogeneous, coprime}.
\]
Throughout the local arguments, we work with a \emph{normalized integral homogeneous lift}
\[
\Phi=[F:G],\qquad F,G\in\OO_K[X,Y]_d,
\]
meaning that $F$ and $G$ have no common factor over $K$ and at least one coefficient of $F$ or $G$ is a unit (equivalently, the content of the pair $(F,G)$ is a unit). Any integral lift can be normalized by dividing through by the common content.

The reductions $\widetilde F,\widetilde G\in k[X,Y]_d$ are not both identically zero (by normalization), so they define a rational map
\[
[\widetilde F:\widetilde G]\colon \PP^1_k\dashrightarrow \PP^1_k
\]
on the complement of the proper closed locus $V(\widetilde F,\widetilde G)$. Since $\PP^1_k$ is a regular curve and the target $\PP^1_k$ is proper, this rational map extends uniquely to a morphism
\[
\widetilde\varphi:\PP^1_k\longrightarrow\PP^1_k,
\]
the \emph{canonical residual morphism}. Computationally, writing
\[
\widetilde F=h\,F_0,\qquad \widetilde G=h\,G_0,\qquad h:=\gcd(\widetilde F,\widetilde G),
\]
one has $\widetilde\varphi=[F_0:G_0]$. We put
\[
e:=\deg(\widetilde\varphi)\le d.
\]
After base change to $\bar k$, this construction gives the corresponding geometric residual morphism, and the degree remains $e$.

We say $\varphi$ has \emph{strict good reduction} at $v$ if $e=d$, equivalently $\Res(F,G)\in\OO_K^\times$; see, e.g., \cite[Theorem~2.15]{SilvermanADS}.

We write $\Res(F,G)$ for the resultant of homogeneous coprime forms $F,G$ defining $\varphi$,
$\Disc(f)$ for the discriminant of a separable polynomial $f$, and $\deg(\tilde\varphi)$ for the
degree of the reduction of $\varphi$ in any integral model. Reductions of maps/polynomials are
denoted by a tilde (e.g. $\tilde\varphi$, $\tilde P$), and reductions of points by a bar (e.g. $\bar x$).
For iterates, we consistently write $\tilde\varphi^{\,m}(\Crit(\tilde\varphi))$.
We use $\Br(\cdot)$ for branch loci, and $\varprojlim$, $\varinjlim$ for projective/direct limits.

\subsection{Preimage towers and Galois action}\label{subsec:towers}
For $x\in \PP^1(\overline K)$ and $n\ge0$, define
\[
X_n(x)=\{y\in \PP^1(\overline K): \varphi^n(y)=x\}.
\]
If $\varphi^n(T)=P_n(T)/Q_n(T)$ with $P_n,Q_n\in K[T]$ coprime, then $X_n(x)$ are the roots of $F_{n,x}(T)=P_n(T)-x\,Q_n(T)\in K(x)[T]$ (in particular, $P_1,Q_1$ represent $\varphi$ at level $n=1$). The Galois group $G_{K(x)}=\Gal(\overline K/K(x))$ acts by permuting $X_n(x)$. The transition maps
\[
\pi_{n+1}:X_{n+1}(x)\to X_n(x),\quad \xi\mapsto \varphi(\xi),
\]
are $G_{K(x)}$-equivariant, defining the projective limit $X_\infty(x)=\varprojlim_n X_n(x)$ and the representation
\[
\rho_x: G_{K(x)}\longrightarrow \Aut(X_\infty(x)).
\]

\subsection{Functoriality along a forward orbit}\label{subsec:functoriality}
Let $x\in \PP^1(\overline K)$ and $y=\varphi(x)$. With the convention of \S\ref{subsec:towers}, an element of $X_\infty(x)$ is a compatible backward orbit
\[
(x,\xi_1,\xi_2,\ldots),\qquad \varphi(\xi_1)=x,\quad \varphi(\xi_{i+1})=\xi_i\ \ (i\ge 1),
\]
whose initial coordinate is forced to be $x$ (since $X_0(x)=\{x\}$). The backward tree above $x$ is naturally identified with the branch of the backward tree above $y=\varphi(x)$ rooted at $x$. Explicitly, there is a canonical injective map
\[
\iota_{x\to y}:X_\infty(x)\hookrightarrow X_\infty(y),\qquad (x,\xi_1,\xi_2,\ldots)\longmapsto (y,x,\xi_1,\xi_2,\ldots).
\]
Equivalently, at finite level, if $x_n=\varphi^n(x)$, then for every $m\ge 0$ there is a natural inclusion
\[
X_m(x_n)\hookrightarrow X_{m+1}(x_{n+1}),
\]
since $\varphi^m(z)=x_n$ implies $\varphi^{m+1}(z)=\varphi(x_n)=x_{n+1}$.
If the forward orbit $O^+(x)=\{x_n\}_{n\ge 0}$ is defined over a field $L$ (in particular over $K$, the case relevant for the local criterion below), these branch inclusions are $G_L$-equivariant and yield a direct system $\{X_\infty(x_n)\}_{n\ge 0}$ of rooted preimage trees along the forward orbit.

\subsection{The orbital preimage tree and orbital arboreal image}\label{subsec:colimit}
Let $x\in\PP^1(K)$ and set $x_n=\varphi^n(x)$. The branch inclusions of \S\ref{subsec:functoriality} define a direct system of rooted preimage trees
\[
X_\infty(x_0)\hookrightarrow X_\infty(x_1)\hookrightarrow X_\infty(x_2)\hookrightarrow\cdots,
\]
and we define the \emph{orbital preimage tree} of the forward orbit $O^+(x)$ by
\[
T_{O^+(x)}\ :=\ \varinjlim_{n}X_\infty(x_n).
\]
Thus $T_{O^+(x)}$ is obtained by gluing the backward preimage trees along the forward orbit, identifying the tree above $x_n$ with the branch of the tree above $x_{n+1}$ rooted at $x_n$. Since the orbit is $K$-rational, $G_K$ acts compatibly on each $X_\infty(x_n)$, hence on the colimit. We obtain the \emph{orbital arboreal representation}
\[
\rho_{O^+(x)}:G_K\longrightarrow \Aut(T_{O^+(x)}),
\]
and its image, the \emph{orbital arboreal image},
\[
\mathcal{G}_{O^+(x)}\ :=\ \operatorname{Im}\bigl(\rho_{O^+(x)}\bigr)\ \subseteq\ \Aut(T_{O^+(x)}).
\]
We refer to the construction $\varinjlim_n X_\infty(x_n)$ as the \emph{orbital colimit}; this term refers to the colimit of the underlying preimage trees, not to a direct limit of pointwise Galois images. For the purposes of this paper, we use only the induced action on finite preimage levels; no additional topological structure on the orbital image is needed.

\begin{proposition}[Equivariance and universal property of the orbital colimit]\label{prop:universal-T}
Assume the forward orbit $O^+(x)\subset\PP^1(K)$, and set $x_n=\varphi^n(x)$. The branch inclusions
\[
\iota_n:X_\infty(x_n)\hookrightarrow X_\infty(x_{n+1}),\qquad (x_n,\xi_1,\xi_2,\ldots)\longmapsto (x_{n+1},x_n,\xi_1,\xi_2,\ldots),
\]
are $G_K$-equivariant. Hence the orbital preimage tree
\[
T_{O^+(x)}\ =\ \varinjlim_n X_\infty(x_n)
\]
is the colimit of this directed system in the category of $G_K$-sets. Equivalently, for every $G_K$-set $Y$, composition with the structure maps $j_n:X_\infty(x_n)\to T_{O^+(x)}$ induces a natural bijection
\[
\Hom_{G_K}\bigl(T_{O^+(x)},Y\bigr)\ \simeq\ \varprojlim_n \Hom_{G_K}\bigl(X_\infty(x_n),Y\bigr),
\]
where the transition maps on the right are given by precomposition with the branch inclusions $\iota_n$.
\end{proposition}
\begin{proof}
Since $x_n\in K$ for all $n$, each $x_n$ is fixed by $G_K$. The branch inclusions commute with the natural $G_K$-action on compatible backward orbits, hence form a directed system in $G_K$-sets. The displayed bijection is precisely the universal property of the filtered colimit in $G_K$-sets.
\end{proof}

\subsection{Two technical lemmata}\label{subsec:lemmata}
\begin{lemma}[Strict good reduction persists under iteration]\label{lem:unit-resultant}
If $\varphi$ has strict good reduction, then for every $n\ge1$ the iterate $\varphi^n$ has strict good reduction. Equivalently, for every $n\ge1$ there exists an integral presentation $(P_n,Q_n)$ of $\varphi^n$ with
\[
\Res(P_n,Q_n)\in\OO^\times.
\]
\end{lemma}
\begin{proof}
By \S\ref{subsec:goodred}, strict good reduction is equivalent to $\Res(F,G)\in\OO^\times$. This unit-resultant condition is stable under iteration (Benedetto~\cite[Theorem~B]{Ben01}): for every $n\ge1$ there exists an integral presentation $(P_n,Q_n)$ of $\varphi^n$ with $\Res(P_n,Q_n)\in\OO^\times$, hence $\varphi^n$ has strict good reduction.
\end{proof}

\begin{lemma}[Unit discriminant $\Rightarrow$ unramified]\label{lem:disc-unram}
Let $f(T)\in\OO[T]$ be a separable polynomial with unit leading coefficient. If $\Disc(f)\in\OO^\times$, then the splitting field of $f$ over $K$ is unramified. In particular, for any set $Y$ of roots of $f$, $K(Y)/K$ is unramified.
\end{lemma}
\begin{proof}
Since $K$ is a local field, its residue field $k$ is perfect. $\Disc(f)\in\OO^\times$ implies the reduction $\tilde f \in k[T]$ is separable. Moreover, the unit leading coefficient ensures that $\deg \tilde f=\deg f$, so the reduction has the correct degree and Dedekind's criterion applies directly. By Dedekind's criterion, this implies the different is trivial, and thus the prime $\mathfrak m$ does not ramify in the splitting field; see, e.g., \cite[Ch.~III, §5, Cor.~1]{SerreLF}. Moreover, a compositum of unramified extensions remains unramified, so the full splitting field of $f$ over $K$ is unramified as well.
\end{proof}

\begin{definition}[Postcritical set]\label{def:pcset}
For the reduction $\widetilde\varphi:\PP^1_k\to\PP^1_k$, define the reduced postcritical set
\[
\PC(\widetilde\varphi)\ :=\ \bigcup_{m\ge 1}\widetilde\varphi^{\,m}\!\big(\Crit(\widetilde\varphi)\big)\ \subset \PP^1_k.
\]
\end{definition}

\begin{remark}\label{rem:critdivisor}
Equivalently, $\Crit(\widetilde\varphi)$ is the ramification divisor; in homogeneous coordinates it is cut out by $(F_0)_X(G_0)_Y-(F_0)_Y(G_0)_X$ (subscripts denote partial derivatives), for the reduced homogeneous presentation $\widetilde\varphi([X:Y])=[F_0:G_0]$ obtained after cancelling the residual common factor $h=\gcd(\widetilde F,\widetilde G)$.
\textbf{(Corrected in v4.)}
When $p\mid e$, the Wronskian $(F_0)_X(G_0)_Y-(F_0)_Y(G_0)_X$ may vanish identically; in that case, $\Crit(\widetilde\varphi)$ should be defined intrinsically as the support of the sheaf of relative differentials~$\Omega_{\widetilde\varphi}$, and the degree of the ramification divisor is $2e-2$ for a separable morphism of degree~$e$.
\end{remark}

\begin{proposition}[\'Etale fibers for all iterates $\Leftrightarrow$ non-postcritical]\label{prop:etale-iff-PC}
Let $\widetilde\varphi:\PP^1_k\to\PP^1_k$ be a separable morphism of positive degree, with $k$ perfect.
For $\bar x\in\PP^1_k$, the following are equivalent:
\begin{enumerate}[label=(\alph*)]
\item For every $n\ge1$, the fiber of $\widetilde\varphi^{\,n}$ over $\bar x$ is étale.
\item $\bar x\notin\PC(\widetilde\varphi)=\bigcup_{m\ge1}\widetilde\varphi^{\,m}\!\big(\Crit(\widetilde\varphi)\big)$.
\end{enumerate}
\end{proposition}
\begin{proof}
Recall the standard chain rule for branch loci on $\PP^1$:
\[
\Crit\!\big(\widetilde\varphi^{\,n}\big)
=\bigcup_{j=0}^{n-1}(\widetilde\varphi^{\,j})^{-1}\!\big(\Crit(\widetilde\varphi)\big),
\]
\[
\Br\!\big(\widetilde\varphi^{\,n}\big)
=\widetilde\varphi^{\,n}\!\big(\Crit(\widetilde\varphi^{\,n})\big)
=\bigcup_{m=1}^{n}\widetilde\varphi^{\,m}\!\big(\Crit(\widetilde\varphi)\big)
\ \subset\ \PC(\widetilde\varphi).
\]
Since $k$ is perfect, étaleness of the fiber over $\bar x$ for $\widetilde\varphi^{\,n}$ is equivalent to
$\bar x\notin\Br(\widetilde\varphi^{\,n})$.
($\neg$(b)$\Rightarrow\neg$(a)) If $\bar x\in\PC(\widetilde\varphi)$, then $\bar x\in\Br(\widetilde\varphi^{\,m})$
for some $m\ge1$, so the fiber of $\widetilde\varphi^{\,m}$ over $\bar x$ is not étale.
(a)$\Rightarrow$(b) contrapositively: if for some $n\ge1$ the fiber of $\widetilde\varphi^{\,n}$ over $\bar x$ is not étale, then $\bar x\in\Br(\widetilde\varphi^{\,n})\subset\PC(\widetilde\varphi)$. The set $\PC(\widetilde\varphi)$ is finite because $\Crit(\widetilde\varphi)$ is a finite set (the zero locus of a homogeneous form of degree $2d-2$) and, over the finite residue field $k$, each of its points has finite forward orbit.
\end{proof}

\begin{remark}\label{rem:openlocus}
In particular, $U_k:=\PP^1_k\setminus\PC(\widetilde\varphi)$ is a nonempty Zariski open subset
on which every fiber of $\widetilde\varphi^{\,n}$ is étale for all $n\ge1$.
\end{remark}

\begin{remark}\label{rem:insep}
If $\widetilde\varphi$ is purely inseparable, then every iterate is ramified everywhere and
$\PC(\widetilde\varphi)=\PP^1_k$, so the statement is vacuous and consistent with the proposition.
\end{remark}

\section{Main Result}\label{sec:main}

Throughout this section and the remainder of the paper we work in the \text{tame case}, assuming that the residue characteristic $p$ does not divide the degree of the map $d$ (i.e., $p\nmid d$). This excludes the most immediate inseparability phenomena in the degree-$d$ fibers considered below; the wild case may involve inseparability and additional ramification breaks and is not treated here.

Good reduction is characterized not by all points, but specifically by those integral points whose fibers avoid the reduced postcritical locus.

\begin{theorem}[Galois criterion --- strict local version]\label{thm:main}
Let $K$ be a non-archimedean local field with ring of integers $\OO_K$, maximal ideal $\mathfrak{m}$, and residue field $k$ of characteristic $p$. Let $\varphi\in K(z)$ have degree $d\ge 2$, and assume $p\nmid d$. Let $\widetilde\varphi:\PP^1_k\to\PP^1_k$ be the canonical residual morphism associated to a normalized integral lift, with $e=\deg(\widetilde\varphi)\le d$ as in \S\ref{subsec:goodred}. The following are equivalent:
\begin{enumerate}[label=(\arabic*)]
\item $\varphi$ has strict good reduction (equivalently, $\Res(F,G)\in\OO_K^\times$).
\item $e=d$.
\item There exists a nonempty Zariski open subset
\[
U_k\subseteq \PP^1_k,\qquad U_k\cap\PC(\widetilde\varphi)=\varnothing,
\]
such that the restricted morphism
\[
\widetilde\varphi^{-1}(U_k)\longrightarrow U_k
\]
is finite \'etale of degree $d$. Equivalently, for every geometric point $\bar x\to U_k$, the fiber $\widetilde\varphi^{-1}(\bar x)$ has degree $d$ and is étale (Proposition~\ref{prop:etale-iff-PC}).
\end{enumerate}

\smallskip
\noindent\emph{Consequence.} Under these equivalent conditions, for every finite point $x\in\OO_K$ with reduction $\bar x\in U_k$, $\bar x\neq\tilde\varphi^n(\infty)$, and every $n\ge 1$, the level-$n$ fiber polynomial
\[
F_{n,x}(T)=P_n(T)-x\,Q_n(T)\in\OO_K[T]
\]
from \S\ref{subsec:towers} has unit leading coefficient $\operatorname{lc}(F_{n,x})$ and unit discriminant:
\[
\operatorname{lc}(F_{n,x})\in\OO_K^\times,\qquad \Disc(F_{n,x})\in\OO_K^\times.
\]
Hence $K(X_n(x))/K$ is unramified, and the inertia group $I_v$ acts trivially on $X_\infty(x)$.
\end{theorem}

\begin{remark}[Resultant and discriminants detect different phenomena]\label{rem:two-invariants}
The criterion separates two local invariants of the model. The resultant $\Res(F,G)$ detects residual degree drop, equivalently the presence of base points in the residual rational map $[\widetilde F:\widetilde G]$. The fiber discriminants $\Disc(F_{n,x})$ detect étaleness of the corresponding residual fibers \emph{once the canonical residual morphism has full degree $e=d$}. Thus the resultant controls the global degree of the residual morphism, whereas the discriminants control fiberwise separability. Neither invariant alone suffices: fiber discriminants may be unit on some fibers even when the canonical residual morphism has lower degree. This can occur when $\widetilde F$ and $\widetilde G$ share a nonconstant common factor in $k[X,Y]$.
\end{remark}

\begin{corollary}[Moduli version, $\mathrm{PGL}_2(K)$-invariant]\label{cor:moduli}
In the usual moduli sense, i.e.\ up to $K$-rational $\mathrm{PGL}_2$-conjugacy, the criterion is expressed via the minimal resultant. Let $R_v(\varphi):=\min_{M\in\mathrm{PGL}_2(K)} v(\mathrm{Res}(\varphi^M))$ be the valuation of the minimal resultant. Then $R_v(\varphi)=0$ if and only if there exists a conjugate $M\in\mathrm{PGL}_2(K)$ such that the canonical residual morphism of $\varphi^M$ has degree $d$. Equivalently, condition (3) of Theorem~\ref{thm:main} holds for $\varphi^M$.
\end{corollary}
\begin{proof}
By definition, $R_v(\varphi)=0$ if and only if there exists $M\in\mathrm{PGL}_2(K)$ with
$\Res(\varphi^M)\in\OO_K^\times$ and $\deg\widetilde{\varphi^M}=d$, i.e., $\varphi^M$ has strict good reduction. Applying Theorem~\ref{thm:main} to $\varphi^M$ shows that condition \emph{(2)} holds precisely in this case. This is equivalent to $R_v(\varphi)=0$, as claimed.
\end{proof}

\begin{remark}[Relative inertia]\label{rem:relative_inertia}
Suppose one insists on quantifying over \emph{all} $x \in \PP^1(\overline K)$. In that case, the correct statement involves the relative inertia group $I_{v(x)} \subset G_{K(x)}$ for the appropriate valuation $v(x)$ on $K(x)$. Theorem \ref{thm:main} provides a cleaner criterion by fixing the base field $K$ and restricting the test points $x$ to the unramified integral locus.
\end{remark}

\begin{remark}[Wild case $p \mid d$]
If $p \mid d$, the situation is substantially more complicated. Inseparability in the reduction $\tilde\varphi$ may force the postcritical set to be the entire line ($\PC(\tilde\varphi)=\PP^1_k$), making our criterion vacuous. The wild case requires a separate, deeper analysis of ramification breaks and is beyond the scope of this paper.
\end{remark}

\section{Orbital open image: a canonical formulation}\label{sec:orbital-open-image}

In the classical (pointwise) framework, one fixes a basepoint $x\in\PP^1(\overline K)$ and studies the arboreal representation
\[
\rho_x:\ G_{K(x)}\longrightarrow \Aut\!\big(X_\infty(x)\big),
\]
asking whether the image is ``large'' (e.g., open of finite index) in a natural automorphism group (typically the automorphism group of a regular $d$-ary rooted tree when no local branching obstructions arise).

\begin{remark}[Intuition]\label{rem:orbital_intuition}
Heuristically, the orbital construction assembles the backward preimage trees of points along the forward orbit into a single canonical Galois object (canonical in the sense of the universal property of Proposition~\ref{prop:universal-T}). Each tree $X_\infty(x_n)$ embeds as a branch of $X_\infty(x_{n+1})$ via the inclusion $\iota_{x_n\to x_{n+1}}$, so the orbital tree $T_{O^+(x)}$ is the direct limit of these preimage trees, and the orbital arboreal image $\mathcal{G}_{O^+(x)}$ is the image of the induced Galois action on $T_{O^+(x)}$. This eliminates basepoint dependence while retaining all dynamical ramification data.
\end{remark}

\begin{proposition}[Canonicity and independence of basepoint]\label{prop:canonic-orbital}
Let $x\in\PP^1(K)$ and set $x_n=\varphi^n(x)$. Then:
\begin{enumerate}[label=(\alph*)]
\item $K(O^+(x))=K(O^+(x_n))$ for all $n\ge0$; hence the absolute Galois group of $K(O^+(x))$ is canonically attached to the orbit.
\item The orbital preimage tree $T_{O^+(x)}=\varinjlim_n X_\infty(x_n)$ is independent (up to canonical isomorphism) of the choice of basepoint along the orbit; consequently, $\mathcal{G}_{O^+(x)}=\operatorname{Im}(G_K\to\Aut(T_{O^+(x)}))$ is independent of the choice of basepoint as well.
\item If $M\in\mathrm{PGL}_2(K)$, then $M$ induces a $G_K$-equivariant isomorphism $T_{O^+(x)}\xrightarrow{\sim}T_{O^+(Mx)}$, identifying the corresponding orbital arboreal images up to conjugacy. In particular, the \emph{conjugacy class over $K$} of $\mathcal{G}_{O^+(x)}$ is a conjugacy-invariant of the orbit in moduli.
\end{enumerate}
\end{proposition}
\begin{proof}
(a) For every $m\ge 0$ one has $\varphi^{n+m}(x)=\varphi^{m}(x_n)$, hence
\[
K\big(O^{+}(x)\big)=\bigcup_{m\ge 0}K\!\big(\varphi^{m}(x)\big)
=\bigcup_{m\ge 0}K\!\big(\varphi^{m}(x_n)\big)
=K\big(O^{+}(x_n)\big).
\]
(b) Replacing $x$ by $x_r=\varphi^r(x)$ replaces the directed system
\[
X_\infty(x_0)\hookrightarrow X_\infty(x_1)\hookrightarrow X_\infty(x_2)\hookrightarrow\cdots
\]
by its cofinal tail
\[
X_\infty(x_r)\hookrightarrow X_\infty(x_{r+1})\hookrightarrow X_\infty(x_{r+2})\hookrightarrow\cdots.
\]
By the universal property of filtered colimits in $G_K$-sets (Proposition~\ref{prop:universal-T}), the induced map $T_{O^+(x)}\to T_{O^+(x_r)}$ is a $G_K$-equivariant isomorphism. The orbital arboreal image $\mathcal{G}_{O^+(x)}$, being the image of $G_K$ in $\Aut(T_{O^+(x)})$, is therefore identified canonically with $\mathcal{G}_{O^+(x_r)}$ along this isomorphism.
(c) For $M\in\mathrm{PGL}_2(K)$, write $\varphi^{M}=M\circ\varphi\circ M^{-1}$. Then $O^{+}(Mx)=M\!\big(O^{+}(x)\big)$. At each level $m$, $M$ induces a bijection $X_\infty(x_m)\xrightarrow{\sim}X_\infty((Mx)_m)$ that intertwines the $G_K$-actions and respects the branch inclusions. Passing to the colimit yields a $G_K$-equivariant isomorphism $T_{O^+(x)}\xrightarrow{\sim}T_{O^+(Mx)}$, which identifies the orbital arboreal images up to conjugacy.
\end{proof}

\subsection{The target orbital symmetry group}\label{subsec:target_group}
When no critical collisions occur in the backward trees along the orbit (the ``regular'' situation), each branch inclusion $X_\infty(x_n)\hookrightarrow X_\infty(x_{n+1})$ identifies $X_\infty(x_n)$ with the subtree of $X_\infty(x_{n+1})$ rooted at $x_n$. In this case the finite backward trees along the orbit are each, abstractly, isomorphic to the standard $d$-ary rooted tree $T_d$, but \emph{not} canonically: the colimit $T_{O^+(x)}$ inherits the distinguished family of branch inclusions along the forward orbit, which is data not present in $T_d$.

\textbf{Correction (v4).} Each orbit point~$x_m$ has $d$~preimages, one of which ($x_{m-1}$ for $m\ge1$) lies on the orbit spine, so every vertex of the colimit has degree~$d+1$. The unrooted tree $T_{O^+(x)}$ is therefore $(d+1)$-regular, not $d$-ary; $\Aut(T_{O^+(x)})$ is locally compact but not profinite. The natural profinite target for the Galois action is the spine-fixing subgroup
\[
\Aut_{\mathrm{spine}}(T_{O^+(x)})
:=\bigl\{\sigma\in\Aut(T_{O^+(x)}):\sigma(x_m)=x_m\ \forall\, m\ge0\bigr\},
\]
which decomposes as
$\bigl(\Aut(T_d)\wr S_d\bigr)
\times\prod_{m\ge1}\bigl(\Aut(T_d)\wr S_{d-1}\bigr)$.
Each factor is profinite, so $\Aut_{\mathrm{spine}}$ is profinite and the equivalence between open image and finite-index image is restored.

In the presence of critical preimages, branching at critical points imposes additional relations.

\subsection{The orbital Galois assignment and unramifiedness}\label{subsec:sheaf}
Unless otherwise stated, orbits are considered in $\PP^1(\overline K)$.
For the local criterion, we restrict to integral representatives
$x\in\PP^1(\OO_K)$ with $\bar x\notin\PC(\tilde\varphi)$.
Let $\mathcal{O}_\varphi$ be the \emph{orbit groupoid} of $\varphi$, whose objects are forward orbits $O^+(x)$ in $\PP^1(\overline K)$ and whose arrows are orbit isomorphisms induced by $\varphi$ (e.g.\ base changes $x\mapsto \varphi^n(x)$) and by transport $x\mapsto \sigma(x)$ with $\sigma\in G_K$.
The orbital construction defines an \emph{orbital Galois assignment}
\[
\mathscr{G}_\varphi:\ \mathcal{O}_\varphi\longrightarrow \mathbf{Grp},\qquad
O^+(x)\longmapsto \mathcal{G}_{O^+(x)}=\operatorname{Im}\bigl(G_K\to\Aut(T_{O^+(x)})\bigr),
\]
which is functorial under orbit base change and under $G_K$-transport: an arrow $O^+(x)\to O^+(\varphi^n(x))$ corresponds to the canonical identification of orbital trees of Proposition~\ref{prop:canonic-orbital}(b), and an arrow $O^+(x)\to O^+(\sigma(x))$ corresponds to transport by $\sigma\in G_K$. The fiber over $O^+(x)$ is the orbital arboreal image defined in Section~\ref{subsec:colimit}.

\begin{definition}[Residually admissible orbit and the safe locus]\label{def:residually-admissible}
The residual étale locus $U_k=\PP^1_k\setminus\PC(\tilde\varphi)$ is not, in general, forward-invariant under $\tilde\varphi$: a point $\bar y\in U_k$ may satisfy $\tilde\varphi(\bar y)\in\PC(\tilde\varphi)$. We therefore single out the maximal forward-invariant subset of $U_k$,
\[
U_k^{\mathrm{safe}}\ :=\ \bigcap_{m\ge 0}\tilde\varphi^{-m}(U_k)\ \subseteq\ U_k,
\]
i.e.\ the set of $\bar y\in U_k$ such that $\tilde\varphi^m(\bar y)\in U_k$ for every $m\ge 0$. We say that an integral point $x\in\PP^1(\OO_K)$ is \emph{residually admissible} if $\bar x\in U_k^{\mathrm{safe}}$, equivalently if $\tilde\varphi^m(\bar x)\in U_k$ for every $m\ge 0$. When $\varphi$ has strict good reduction, the identity $\overline{\varphi^m(x)}=\tilde\varphi^m(\bar x)$ holds, so this is equivalent to $\bar x_m\in U_k$ for all $m\ge 0$. The associated forward orbit $O^+(x)$ is then called a \emph{residually admissible orbit}.
\end{definition}

\begin{remark}[Why the safe locus is needed]\label{rem:why-safe}
The locus $U_k^{\mathrm{safe}}$ keeps the entire residual forward orbit inside $U_k$, where Theorem~\ref{thm:main} provides the level-$n$ unit-discriminant statement at each basepoint along the orbit. A simple example illustrates that this strengthening is genuine: take $K=\QQ_3$ and $\varphi(z)=z^2+1$, so that $\tilde\varphi(z)=z^2+1\in\mathbb{F}_3[z]$. In homogeneous coordinates the ramification divisor is cut out by $4XY\equiv XY\pmod 3$, so $\Crit(\tilde\varphi)=\{\bar 0,\infty\}$; the forward orbits give $\tilde\varphi(\bar 0)=\bar 1$, $\tilde\varphi(\bar 1)=\bar 2$, $\tilde\varphi(\bar 2)=\bar 2$, and $\tilde\varphi(\infty)=\infty$, so $\PC(\tilde\varphi)=\{\bar 1,\bar 2,\infty\}$ and $U_k=\{\bar 0\}\subset\PP^1_{\mathbb{F}_3}$. The point $\bar 0\in U_k$, but $\tilde\varphi(\bar 0)=\bar 1\in\PC(\tilde\varphi)$, so the residual forward orbit of $\bar 0$ exits $U_k$ at the first iterate; in particular $\bar 0\notin U_k^{\mathrm{safe}}$. Forward-invariance is therefore a genuine extra condition, and Definition~\ref{def:sheaf_unram} and Lemma~\ref{lem:uniform} below are formulated on the safe locus accordingly.
\end{remark}

\begin{definition}[Unramified at $v$ over the safe locus]\label{def:sheaf_unram}
Let $v$ be a non-archimedean place of $K$ with inertia $I_v\subset G_K$.
We say that the orbital Galois assignment $\mathscr{G}_\varphi$ is \emph{unramified at $v$ over the safe integral locus} if, for every residually admissible $x\in\PP^1(\OO_K)$ (i.e.\ $\bar x\in U_k^{\mathrm{safe}}$), the inertia subgroup has trivial image on the orbital tree of $x$,
\[
\rho_{O^+(x)}(I_v)\ =\ 1\quad\text{in }\Aut(T_{O^+(x)}).
\]
\end{definition}
This definition does not depend on the choice of a representative of the orbit: residual admissibility and the orbital tree are both invariant under $x\mapsto\varphi^n(x)$ (Proposition~\ref{prop:canonic-orbital}).

\begin{lemma}[Uniform inertia triviality on the safe locus]\label{lem:uniform}
If $\varphi$ has strict good reduction, then $\rho_{O^+(x)}(I_v)=1$ for every residually admissible $x\in\PP^1(\OO_K)$.
\end{lemma}
\begin{proof}
Let $x\in\PP^1(\OO_K)$ with $\bar x\in U_k^{\mathrm{safe}}$, and write $x_m=\varphi^m(x)$, so that $\bar x_m=\tilde\varphi^m(\bar x)\in U_k$ for every $m\ge 0$ by Definition~\ref{def:residually-admissible}. For each fixed $m$, the basepoint $x_m\in\PP^1(\OO_K)$ has $\bar x_m\in U_k$, so Theorem~\ref{thm:main}, applied at the basepoint $x_m$, gives that $I_v$ acts trivially on $X_n(x_m)$ for every $n\ge 0$. Hence $I_v$ acts trivially on the inverse-limit preimage tree $X_\infty(x_m)=\varprojlim_n X_n(x_m)$, and the branch inclusions $X_\infty(x_m)\hookrightarrow X_\infty(x_{m+1})$ are $I_v$-equivariant. Passing to the colimit, $I_v$ acts trivially on $T_{O^+(x)}=\varinjlim_m X_\infty(x_m)$, so $\rho_{O^+(x)}(I_v)=1$ in $\Aut(T_{O^+(x)})$.
\end{proof}

For each residually admissible $x\in\PP^1(\OO_K)$ and integers $m\ge 0$, $n\ge 1$, the inverse-preimage set $X_n(x_m)\subset\PP^1(\overline K)$ consists of the roots of $F_{n,x_m}\in K[T]$ (regarded projectively, accounting for the point at infinity when $\operatorname{lc}(F_{n,x_m})=0$); when these roots lie in $\PP^1(\overline\OO_K)$ they reduce to $\PP^1(\overline k)$, giving the \emph{orbital reduction map}
\[
\mathrm{red}_{m,n}\colon X_n(x_m)\longrightarrow \widetilde{X}_n(\bar x_m)\subseteq\PP^1(\bar k),
\]
$G_K$-equivariantly through the inertia action on the source. Under strict good reduction, $\widetilde{X}_n(\bar x_m)=\tilde\varphi^{-n}(\bar x_m)$; in general, the inclusion may be strict.

\begin{proposition}[Strict good reduction $\Rightarrow$ bijective orbital specialization]\label{prop:specialization-from-SGR}
Assume $\varphi$ has strict good reduction. Then for every residually admissible $x\in\PP^1(\OO_K)$, every $m\ge 0$ and every $n\ge 1$, the orbital reduction map $\mathrm{red}_{m,n}\colon X_n(x_m)\to\widetilde X_n(\bar x_m)$ is a $G_K$-equivariant bijection. In concrete terms, $|X_n(x_m)|=|\widetilde X_n(\bar x_m)|=d^n$ and the reduction is one-to-one at every level.
\end{proposition}
\begin{proof}
\textbf{(Corrected in v4.)} Fix $m\ge 0$ and $n\ge 1$. Since $x\in\PP^1(\OO_K)$ and $\varphi$ has strict good reduction, $x_m=\varphi^m(x)\in\PP^1(\OO_K)$. The projective fibre $\PP^1_R\times_{\Phi^n}x_m$ is a finite \'etale $R$-scheme of rank~$d^n$ (see the self-contained argument in the corrigendum). In particular, both $X_n(x_m)$ and $\widetilde X_n(\bar x_m)$ have cardinality~$d^n$, and the splitting field $K(X_n(x_m))/K$ is unramified. When $x_m\in\OO_K$ and the affine chart condition $\bar x_m\neq\tilde\varphi^n(\infty)$ holds, Theorem~\ref{thm:main} gives $\operatorname{lc}(F_{n,x_m})\in\OO_K^\times$ and $\Disc(F_{n,x_m})\in\OO_K^\times$, recovering the same conclusion via the polynomial $F_{n,x_m}$. Over a Henselian discrete valuation ring, the reduction map of a finite separable \'etale algebra is a bijection on geometric points; this gives the bijectivity of $\mathrm{red}_{m,n}$. $G_K$-equivariance is immediate from the construction.
\end{proof}

\begin{proposition}[Bijective specialization at level $1$ $\Rightarrow$ strict good reduction]\label{prop:SGR-from-specialization}
Let $\widetilde\varphi:\PP^1_k\to\PP^1_k$ be the canonical residual morphism associated to a normalized integral lift of $\varphi$. Assume there exists $x_0\in\PP^1(\OO_K)$ such that
\begin{itemize}
\item $X_1(x_0)$ consists of $d$ distinct geometric points over $\overline K$, and
\item the specialization map
\[
\mathrm{red}_{0,1}\colon X_1(x_0)\longrightarrow \bigl(\widetilde\varphi^{-1}(\bar x_0)\bigr)_{\mathrm{red}}
\]
is bijective.
\end{itemize}
Then $\varphi$ has strict good reduction.
\end{proposition}
\begin{proof}
Let $e=\deg(\widetilde\varphi)\le d$. For a morphism $\PP^1_{\bar k}\to\PP^1_{\bar k}$ of degree $e$, every geometric fiber has scheme-theoretic length $e$, hence its underlying reduced set has cardinality at most $e$. By the hypotheses,
\[
d \;=\; |X_1(x_0)| \;=\; \bigl|\bigl(\widetilde\varphi^{-1}(\bar x_0)\bigr)_{\mathrm{red}}\bigr| \;\le\; e \;\le\; d,
\]
so $e=d$. By Theorem~\ref{thm:main} (the equivalence (1)$\Leftrightarrow$(2)), $\varphi$ has strict good reduction.
\end{proof}

\begin{remark}[Pure inertia-triviality does not detect degree drop]\label{rem:5z2+z}
Lemma~\ref{lem:uniform} produces inertia-triviality on the orbital tree as a \emph{consequence} of strict good reduction, but pure inertia-triviality is properly weaker than bijective specialization, and therefore fails to characterize it. The map $\varphi(z)=5z^2+z$ over $K=\QQ_5$ illustrates the subtlety. The integral lift $\Phi=[5X^2+XY:Y^2]$ has reduction $\tilde\Phi=[XY:Y^2]$, which after cancelling the common factor $Y$ defines $\tilde\varphi(z)=z$ of degree $1<2$; hence $\varphi$ has bad reduction. The level-$1$ fiber polynomial at $x_0=0$ is $F_{1,0}(T)=5T^2+T=T(5T+1)$, whose roots $\{0,-1/5\}$ are both $K$-rational, so $I_v$ acts trivially on $X_1(0)$. The reduction map $\mathrm{red}_{0,1}$, however, is \emph{not} bijective: the canonical residual morphism is $\widetilde\varphi(z)=z$, hence $\widetilde\varphi^{-1}(\bar 0)=\{\bar 0\}$ as a set, so the target of $\mathrm{red}_{0,1}$ has cardinality~$1$ while the source has cardinality~$2$. The level-$1$ degree drop of $\tilde\varphi$ is detected by the cardinality mismatch in $\mathrm{red}_{0,1}$, not by the inertia action on $X_1(0)$. This is the conceptual reason why Propositions~\ref{prop:specialization-from-SGR}--\ref{prop:SGR-from-specialization} are stated in terms of the reduction map $\mathrm{red}_{m,n}$ rather than in terms of the inertia subgroup alone. The map $\varphi(z)=5z^2+z$ is the $p=5$ instance of the family $\varphi(z)=pz^2+z$ studied as a bad-reduction example in Section~\ref{sec:examples} (Example~\ref{ex:bad}); the present remark complements that example by exhibiting, at the orbital level, the precise invariant that detects the level-$1$ degree drop of $\tilde\varphi$.
\end{remark}

\begin{remark}[Total space and the $G_K$-action]\label{rem:etale_space}
The disjoint union $\coprod_{O}\mathcal{G}_O$ over residually admissible forward orbits assembles the orbital arboreal images into a $G_K$-equivariant family, with $G_K$ acting by permuting the fibers and transporting structure among them. Lemma~\ref{lem:uniform} expresses inertia-triviality as a necessary condition for strict good reduction, while Propositions~\ref{prop:specialization-from-SGR}--\ref{prop:SGR-from-specialization} provide the genuine arithmetic equivalence in terms of bijective orbital specialization.
\end{remark}

\subsection{Local criteria and openness outside a finite set of places}\label{subsec:adelic}
By Theorem~\ref{thm:main}, for each fixed base point $x$, assuming $\varphi$ is postcritically finite, the set of places $v$ where $K(X_n(x))/K$ is ramified is finite~\cite{Ait04,Bri15,Cul12}. Hence, outside a finite set $S$, the local images are unramified. \textbf{(Corrected in v4.)} Uniformity of~$S$ along an infinite orbit requires further hypotheses; openness reduces to global (Frobenius) permutation statistics in $\Aut_{\mathrm{spine}}(T_{O^+(x)})$ (cf.\ \cite{BostonJones,Stoll,SilvermanADS}).

\paragraph*{Adelic outlook.}
\textbf{(Corrected in v4.)}
If, for a given orbit, there exists a \emph{uniform} finite set~$S$ of places such that
$K(X_n(x_m))/K$ is unramified outside~$S$ for every $m\ge 0$ and $n\ge 1$, then outside~$S$
inertia acts trivially on the fibers and the local images are controlled by Frobenius statistics
in $\Aut_{\mathrm{spine}}(T_{O^+(x)})$.  Under the usual non-special hypotheses, it is natural to
expect the global orbital image to be open in the restricted product of local targets, with
equidistribution governed by Chebotarev-type principles.  The existence of such a uniform~$S$
is known when $\varphi$ is postcritically finite~\cite{Ait04,Bri15,Cul12}; in general, uniformity
along an infinite orbit is an additional hypothesis (cf.\ the caveat above).

\subsection{Classical vs. orbital open image}\label{subsec:open_image}
Classically, one fixes $x$ and asks for open image of $\rho_x$ in $\Aut(X_\infty(x))$ (cf.\ Odoni--Jones; see \cite{JonesBesancon2013}). The drawback is the basepoint dependence. The orbital setting removes it; we record the natural open-image expectation as a question rather than a conjecture, since the orbital target $\Aut_{\mathrm{spine}}(T_{O^+(x)})$ is the genuine intrinsic profinite target and the formulation should be on it directly.

\begin{question}[Orbital Open Image]\label{q:orbital-regular}
Let $K$ be a global field, and let $\varphi\in K(z)$ be a rational map of degree $d\ge 2$ which is \emph{non-special} (i.e.\ not conjugate over $\overline K$ to a power map, Chebyshev, or a Lattés map). Let $O^+(x)$ be a forward orbit which is infinite and such that no $x_n$ lies on the (global) postcritical set $\PC(\varphi)\subset\PP^1(\overline K)$ (here we mean the postcritical set of $\varphi$ over the algebraic closure, distinct from the residual $\PC(\tilde\varphi)\subset\PP^1_k$ used in the local statements above). Is it true that the orbital representation
\[
\rho_{O^+(x)}\colon G_{K(O^+(x))}\longrightarrow \Aut_{\mathrm{spine}}(T_{O^+(x)})
\]
has \emph{open image} (equivalently: image of finite index) in $\Aut_{\mathrm{spine}}(T_{O^+(x)})$?
\end{question}

\begin{remark}[Negative answers to Question~\ref{q:orbital-regular}]\label{rem:negative-Q411}
\textbf{(Added in v4.)}
The answer to Question~\ref{q:orbital-regular} is negative in general; there are two independent obstructions.

\emph{Covering obstruction.}
Let $\varphi(z)=z^3+1$ over $K=\QQ(\zeta_3)$, $x_0\in\OO_K$ with infinite orbit avoiding $\PC(\varphi)$. The map $\tau(z)=\zeta_3 z$ satisfies $\varphi\circ\tau=\varphi$ (since $(\zeta_3 z)^3+1=z^3+1$), so $\tau$ is a deck transformation: $\operatorname{Deck}(\varphi)\simeq\ZZ/3\ZZ$. However, $\tau$ does \emph{not} commute with $\varphi$ (one checks $\tau\circ\varphi\neq\varphi\circ\tau$), so $\Aut_{\mathrm{dyn}}(\varphi,\overline K)=\{1\}$. The deck symmetry forces every lateral preimage $\zeta_3^j x_m$ to lie in $K(x_m)$, so $\Gal(\overline K/K)$ fixes all lateral branches and the orbital image has infinite index in $\Aut_{\mathrm{spine}}(T_{O^+(x_0)})$. The map is not special: the critical orbit $0\mapsto 1\mapsto 2\mapsto 9\mapsto\cdots$ is infinite.

\emph{Postcritical obstruction (trivial deck group).}
Let $f(z)=z^3-\tfrac{3}{2}z^2$ over $\QQ$, $x_0=-1$. The three critical values are distinct, so $\operatorname{Deck}(f)=\{1\}$. The postcritical set $\PC(f)=\{0,-\tfrac{1}{2},\infty\}$ is finite (all fixed or periodic), but the orbit of $x_0$ avoids it: $-1\mapsto -\tfrac{5}{2}\mapsto -25\mapsto\cdots$ with $|f^n(-1)|\to\infty$. Since $f$ is postcritically finite and non-special (ramification signature $(2,2,3)$), the pointwise arboreal image $\operatorname{Im}(\rho_{x_0})\subset\Aut(X_\infty(x_0))$ has infinite index by the criterion of Jones~\cite{JonesBesancon2013}. The surjection $\Aut_{\mathrm{spine}}(T_{O^+(x_0)})\twoheadrightarrow\Aut(X_\infty(x_0))$ propagates this to the orbital setting.

The first example shows that the obstruction survives when $\Aut_{\mathrm{dyn}}(\varphi,\overline K)=\{1\}$; it is the \emph{deck} group, not the dynamical automorphism group, that matters. The second shows that the obstruction persists even when $\operatorname{Deck}(\varphi)=\{1\}$. Compositional structure provides a third source of infinite-index images: if $\varphi=g^2+c$ with $\deg g\ge 2$, $\operatorname{Deck}(\varphi)=\{1\}$, and $\varphi$ is not postcritically finite, the partition of $\varphi^{-1}(y)$ into $g$-blocks forces a Galois-stable block system at every vertex, giving infinite index by an argument analogous to the imprimitive case in classical arboreal representations. A natural restricted form of Question~\ref{q:orbital-regular} therefore asks: \emph{under what conditions on the monodromy and on the arithmetic independence of the lateral branches does the orbital representation have open image?}
\end{remark}

In the regular rooted case, this question specializes to the usual arboreal open-image problem after choosing a compatible labelling of the preimage tree; the orbital formulation introduced above keeps track of the non-canonical nature of such choices along a forward orbit.

\begin{remark}[Regular vs.\ general orbits]\label{rem:general-wreath}
\textbf{(Corrected in v4.)} In the regular case (no critical collisions along the orbit), the colimit $T_{O^+(x)}$ is a $(d+1)$-regular unrooted tree (each orbit vertex has degree~$d+1$). The profinite target $\Aut_{\mathrm{spine}}(T_{O^+(x)})$ contains, for each lateral branch~$B$, a canonical subgroup $\Aut(B)\simeq\Aut(T_d)$ acting trivially on the complement of~$B$. In the presence of critical preimages, $\Aut_{\mathrm{spine}}(T_{O^+(x)})$ encodes the full wreath recursion together with the multiplicity data at the critical branchings. In either case, the natural and orbit-invariant profinite target for the orbital Galois action is $\Aut_{\mathrm{spine}}(T_{O^+(x)})$, which is the form used in Question~\ref{q:orbital-regular}. \textbf{(Corrected in v4.)} The set $T_{O^+(x)}$ carries a natural graph structure: two elements $y,z\in T_{O^+(x)}$ are adjacent whenever $\varphi(y)=z$.  The group $\Aut_{\mathrm{spine}}(T_{O^+(x)})$ acts on this graph, preserving both vertices and edges; the edge structure encodes the map $\varphi$ connecting consecutive preimage levels and is part of the data preserved by graph automorphisms.
\end{remark}

\begin{remark}[Pointwise versus orbital open image]\label{rem:transfer-pointwise-orbital}
The relation between the orbital arboreal image $\mathcal{G}_{O^+(x)}=\operatorname{Im}(G_K\to\Aut(T_{O^+(x)}))$ and the pointwise images $\operatorname{Im}(\rho_{x_n})\subseteq\Aut(X_\infty(x_n))$ is more subtle than a direct transfer. The branch inclusion $X_\infty(x_n)\hookrightarrow T_{O^+(x)}$ exhibits each pointwise tree as a rooted subtree of the orbital tree, and the action of $G_K$ on $T_{O^+(x)}$ records information about how Galois permutes points across sister branches that is not visible in any single pointwise tower. A clean transfer between pointwise and orbital open-image statements therefore requires additional combinatorial hypotheses on the branching pattern along the orbit and is best treated separately.
\end{remark}

\begin{remark}[Special maps and small images]\label{rem:special-small}
\textbf{(Corrected in v4.)} For the special families (power, Chebyshev, Latt\`es), the orbital Galois image $\mathcal{G}_{O^+(x)}$ is typically much smaller than $\Aut_{\mathrm{spine}}(T_{O^+(x)})$ (e.g.\ virtually abelian or coming from endomorphisms of elliptic curves), and one cannot expect open image. The formulation of Question~\ref{q:orbital-regular} isolates the \emph{right} profinite target in all cases via $\Aut_{\mathrm{spine}}(T_{O^+(x)})$.
\end{remark}

\section{Proofs of Theorem \ref{thm:main}}\label{sec:proofs}
\begin{proof}[Proof of $(1)\Leftrightarrow(2)$]
By construction (\S\ref{subsec:goodred}), if $\widetilde F=h F_0$, $\widetilde G=h G_0$ with $h=\gcd(\widetilde F,\widetilde G)$, then $\widetilde\varphi=[F_0:G_0]$ has degree $e=d-\deg h$. Hence $e=d$ iff $\deg h=0$ iff $h\in k^\times$, equivalently iff $\widetilde F,\widetilde G$ are coprime in $k[X,Y]$, equivalently iff $\Res(F,G)\in\OO_K^\times$ (\cite[Theorem~2.15]{SilvermanADS}). This is strict good reduction.
\end{proof}

\begin{proof}[Proof of $(2)\Rightarrow(3)$]
Assume $e=d$, i.e.\ $\widetilde\varphi$ is a morphism $\PP^1_k\to\PP^1_k$ of degree $d$. Since $p\nmid d$ by the standing assumption, the inseparable degree of $\widetilde\varphi$ (a power of $p$ dividing $d$) is $1$, so $\widetilde\varphi$ is separable. The reduced postcritical set $\PC(\widetilde\varphi)$ is finite (Proposition~\ref{prop:etale-iff-PC}), so $U_k:=\PP^1_k\setminus\PC(\widetilde\varphi)$ is a nonempty Zariski open subset.

By Proposition~\ref{prop:etale-iff-PC}, for every geometric point $\bar x\to U_k$ the fiber of $\widetilde\varphi$ over $\bar x$ is étale. Combined with $\deg\widetilde\varphi=d$, this gives that the restricted morphism $\widetilde\varphi^{-1}(U_k)\to U_k$ is finite étale of degree $d$, establishing (3).
\end{proof}

\begin{proof}[Proof of $(3)\Rightarrow(2)$]
Assume there exists $U_k$ as in (3). The restriction $\widetilde\varphi^{-1}(U_k)\to U_k$ has degree $d$, and the degree of a dominant morphism between integral schemes is determined by the function-field extension and is preserved by restriction to nonempty Zariski opens. Hence $\deg\widetilde\varphi=d$, i.e.\ $e=d$.
\end{proof}

\begin{proof}[Proof of the Consequence (unit fiber polynomials and unramifiedness)]
Assume the equivalent conditions (1)--(3) hold, and let $x\in\OO_K$ be a finite point with $\bar x\in U_k$. \textbf{(Corrected in v4.)} By Lemma~\ref{lem:unit-resultant}, $\Res(P_n,Q_n)\in\OO_K^\times$ for all $n\ge 1$, so each iterate has strict good reduction. The projective fibre $\PP^1_R\times_{\Phi^n}x$ is finite \'etale of degree~$d^n$, and the splitting field $K(X_n(x))/K$ is unramified. When additionally the affine chart condition $\bar x\neq\tilde\varphi^n(\infty)$ holds, Proposition~\ref{prop:etale-iff-PC} gives that the reduced fiber $\widetilde F_{n,\bar x}\in k[T]$ has degree $d^n$ and is separable; equivalently $\operatorname{lc}(F_{n,x})\in\OO_K^\times$ and $\Disc(F_{n,x})\in\OO_K^\times$. By Lemma~\ref{lem:disc-unram}, $K(X_n(x))/K$ is unramified, and inertia $I_v$ acts trivially on $X_\infty(x)$.
\end{proof}

\section{Examples over \texorpdfstring{$\QQ_p$}{Qp}}\label{sec:examples}
\begin{example}[Good reduction: $\varphi(z)=z^2-1$ for $p\neq 2$]\label{ex:good1}
Consider $F(X,Y)=X^2-Y^2$ and $G(X,Y)=Y^2$, so that $\varphi$ is given by $z \mapsto z^2-1$. The resultant $\Res(F,G)$ equals $1 \in \ZZ_p^\times$, confirming that $\varphi$ has strict good reduction. The reduction is $\tilde\varphi(z)=z^2-1$ of degree $2$ over $\mathbb{F}_p$. The critical points of $\tilde\varphi$ are $\Crit(\tilde\varphi) = \{0, \infty\}$, and the reduced postcritical set is $\PC(\tilde\varphi) = \{-1, 0, \infty\}$. For any $x \in \ZZ_p$ such that $\bar{x} \notin \{-1, 0, \infty\}$, the reduced level-$1$ fiber is $\tilde{F}_{1,\bar{x}}(T) = T^2 - (\bar{x} + 1)$, which has degree $2$ and is separable over $\mathbb{F}_p$ whenever $p \neq 2$. Explicitly, the fiber discriminant is $\Disc(F_{1,x}) = 4(x + 1)$, which is a $p$-adic unit provided $v_p(x + 1) = 0$. Therefore, $K(X_1(x))/\QQ_p$ is unramified. By Lemma~\ref{lem:unit-resultant}, this property persists at every level: each iterate satisfies $\Disc(F_{n,x}) \in \ZZ_p^\times$, ensuring that $K(X_n(x))/\QQ_p$ remains unramified for all $n \geq 1$. This verifies, through explicit computation, that the inertia group acts trivially on each preimage tower, consistent with the local criterion for good reduction.
\end{example}

\begin{example}[Good reduction: $\varphi(z)=z^2+p$]\label{ex:good2}
Here $F(X,Y)=X^2+pY^2$, $G(X,Y)=Y^2$. Then $\Res(F,G)=F(1,0)^2=1\in\ZZ_p^\times$; thus $\varphi$ has strict good reduction and $\tilde\varphi(z)=z^2$. We have $\PC(\tilde\varphi)=\{0,\infty\}$. If $\bar x\notin\{0,\infty\}$, the reduced fiber $\tilde F_{1,\bar x}(T)=T^2-\bar x$ has degree $2$ and is separable, so $\Disc(F_{1,x})=-4(p-x)\in\ZZ_p^\times$. This aligns with Theorem \ref{thm:main}.
\end{example}

\begin{remark}[Why the naïve NOS fails]\label{rem:naive_fails}
Even when $\varphi$ has strict good reduction, unramified towers need not occur for \emph{all} base points $x\in\PP^1(\overline K)$. The map $\varphi(z)=z^2+p$ over $\QQ_p$ already shows this. In an integral model one has $\tilde\varphi(z)=z^2$ and hence $\PC(\tilde\varphi)=\{0,\infty\}$. If one quantifies over all $x$, ramification can appear for reasons unrelated to the dynamics:
\begin{itemize}
\item If $v_p(x)<0$ (so $x\notin\OO_K$), then $\bar x=\infty\in\PC(\tilde\varphi)$ and $\Disc(F_{1,x})=-4(p-x)\notin\OO_K^\times$; since $v_p(x-p)=v_p(x)$ when $v_p(x)<0$, the level--$1$ extension ramifies if and only if $v_p(x)$ is odd (see below).
\item If $x\in\OO_K$ but $\bar x=0\in\PC(\tilde\varphi)$, then $\widetilde{F_{1,x}}(T)=T^2$ is inseparable and $\Disc(F_{1,x})=-4(p-x)\notin\OO_K^\times$; whether ramification actually occurs depends on the parity of $v_p(x-p)$ (see below).
\end{itemize}
Thus, it is essential in Theorem~\ref{thm:main} to restrict to the integral locus with $\bar x\notin\PC(\tilde\varphi)$, where the reduced fibers are degree-$d$ and étale---equivalently, the fiber discriminants are units and the towers are unramified.

\textbf{(Corrected in v4.)} For $K=\QQ_p$ with $p$ odd and $x\in K\setminus\{p\}$, the extension $K(X_1(x))/K$ ramifies if and only if $v_p(x-p)$ is odd: indeed, $F_{1,x}(T)=T^2+(p-x)$, so $K(X_1(x))=K(\sqrt{x-p})$, which is ramified precisely when $v_p(x-p)$ is odd.  In particular, when $v_p(x)<0$ one has $v_p(x-p)=v_p(x)$, so ramification depends on the parity of $v_p(x)$.  When $x=p$, $F_{1,p}(T)=T^2$ and $K(X_1(p))=K$; the extension is trivial.
\end{remark}

\begin{example}[Non-polynomial with strict good reduction via conjugation]\label{ex:conjugate}
Fix a prime $p$ and consider $\varphi(z)=z^2+p$ over $\QQ_p$, which has strict good reduction
(Example~\ref{ex:good2}). Conjugate by the inversion $M(z)=1/z\in\mathrm{PGL}_2(\OO_{\QQ_p})$ to obtain
\[
\psi=\varphi^{M}=M\circ\varphi\circ M^{-1}(z)=\frac{z^2}{1+p z^2}.
\]
In homogeneous coordinates, $\psi([X:Y])=[X^2:Y^2+pX^2]$ has degree $2$ and
$\Res(X^2,Y^2+pX^2)=1\in\ZZ_p^\times$. Hence $\psi$ has strict good reduction and
$\tilde\psi(z)=z^2$.
Therefore $\PC(\tilde\psi)=\{0,\infty\}$, and for every $x\in\PP^1(\OO_{\QQ_p})$ with
$\bar x\notin\{0,\infty\}$, the reduced fibers of all iterates are étale of degree $2$;
equivalently, the fiber polynomials have unit leading coefficient and unit discriminant, so
$K(X_n(x))/\QQ_p$ is unramified for all $n\ge1$.
This illustrates that the criterion applies uniformly to non-polynomial rational maps.
\end{example}

\begin{example}[Bad reduction: $\varphi(z)=p z^2+z$]\label{ex:bad}
We have $F(X,Y)=pX^2+XY$, $G(X,Y)=Y^2$. The reductions satisfy $\widetilde F=XY$, $\widetilde G=Y^2$, so $h=\gcd(\widetilde F,\widetilde G)=Y$ and the canonical residual morphism is $\widetilde\varphi=[X:Y]$, the identity, of degree $e=1<2=d$. Hence condition (2) of Theorem~\ref{thm:main} fails and $\varphi$ has bad reduction; for any $\bar x$, the fiber of $\widetilde\varphi$ over $\bar x$ is the single point $\bar x$, not a degree-$2$ scheme.
\end{example}

\section{The degree-one case (for completeness)}\label{sec:d1}
Our analysis focuses on $d\ge 2$, where arboreal representations exhibit nontrivial branching and Galois phenomena. For completeness, we briefly address the $d=1$ case.
For $\deg\varphi=1$, every map is a M\"obius transformation
\[
\varphi(z)=\frac{az+b}{cz+d},\qquad M=\begin{pmatrix}a&b\\ c&d\end{pmatrix},\ \det M\ne0.
\]
There are no critical points, hence $\PC(\tilde\varphi)=\varnothing$. The preimage towers are linear:
for any $x\in\PP^1(\overline K)$ and any $n\ge1$, the level-$n$ fiber is a single point defined over $K(x)$, and if $x\in\PP^1(\OO_K)$ then $K(X_n(x))=K$ for all $n$.

\begin{proposition}[Degree one]\label{prop:d1}
Let $K$ be a non-archimedean local field and $\varphi\in K(z)$ have degree $1$ with matrix representative $M$. The following are equivalent:
\begin{enumerate}[label=(\alph*)]
\item $\varphi$ has strict good reduction.
\item $\det M\in\OO_K^\times$ (equivalently, $\Res(F,G)=\det M$ is a unit).
\item There exists a nonempty open $U_k\subset\PP^1_k$ such that $\tilde\varphi^{-1}(U_k)\to U_k$ is finite \'etale of degree~$1$.
\end{enumerate}
Moreover, for every $x\in\PP^1(\OO_K)$ and every $n\ge1$, one has $K(X_n(x))=K$ and $I_v$ acts trivially on $X_\infty(x)$.
\end{proposition}
\begin{proof}
Since $F,G$ are linear forms, $\Res(F,G)=ad-bc=\det M$. Thus (a)$\Leftrightarrow$(b). If $\det M\in\OO_K^\times$,
then $\tilde M\in\mathrm{PGL}_2(k)$ is invertible and the reduced map has degree $1$, whence (c). Conversely, if (c) holds, the reduced fiber has degree $1$ on a Zariski open set, forcing $\deg\tilde\varphi=1$ and hence $\det M\in\OO_K^\times$.
The statement on towers follows from linearity.
\end{proof}

\begin{corollary}[Orbital open image is vacuous for $d=1$]\label{cor:d1-open-image}
If $\deg\varphi=1$, then every finite preimage fiber consists of a single point. Hence the induced Galois action on the orbital preimage tree $T_{O^+(x)}$ is trivial, and the orbital representation
\[
\rho_{O^+(x)}:G_K\to \Aut(T_{O^+(x)})
\]
has trivial image. In particular, the open-image question is vacuous in degree one.
\end{corollary}

\section{Remarks on Context and the \texorpdfstring{$\ell$}{l}-adic module}\label{sec:remarks}
The framework introduced here yields three concrete consequences for subsequent work in both local and global contexts.
\begin{itemize}
\item \text{Orbit-level packaging.} The orbital preimage tree $T_{O^+(x)}$ consolidates the backward preimage trees along a forward orbit into a single canonical orbit-level object, characterized by the universal property of Proposition~\ref{prop:universal-T} as a colimit in $G_K$-sets, and the corresponding orbital arboreal image $\mathcal{G}_{O^+(x)}=\Im(G_K\to\Aut(T_{O^+(x)}))$ is the image of the induced Galois action on this tree. This construction is independent of the basepoint (Proposition~\ref{prop:canonic-orbital}) and interacts effectively with iteration and conjugation (see Section~\ref{sec:orbital-open-image}); these properties are the key structural features for comparing fibers across levels and for formulating global questions.
\item \text{Local criterion on a natural open set.} Formulating the local Galois criterion on the residual open locus $\PP^1_k\setminus\PC(\widetilde\varphi)$ clarifies the role of the reduced postcritical set. \emph{Under strict good reduction}, on this locus the reduced fibers have degree $d$ and are étale, so the discriminants are units and the associated towers are unramified (see Theorem~\ref{thm:main} and the examples). This approach isolates the genuinely dynamical contribution to ramification.
\item \text{Compatibility with global considerations.} The orbital target $\Aut(T_{O^+(x)})$ provides an orbit-invariant formulation of the open-image question (Question~\ref{q:orbital-regular}) and isolates the regular-vs-branching distinction (Remark~\ref{rem:general-wreath}); both ingredients are needed to relate local Frobenius statistics to global open-image statements.
\end{itemize}

\subsection{The arboreal Tate module}\label{subsec:tate}
The arboreal preimage data carries a natural $\ell$-adic Galois module, and the language for describing its $p$-adic counterpart requires care; we record both cases here.
\begin{itemize}
\item \text{Case $\ell \neq p$:} \textbf{(Corrected in v4.)} The transition maps $\pi_{n+1}\colon X_{n+1}(x)\to X_n(x)$, $\xi\mapsto\varphi(\xi)$ (\S\ref{subsec:towers}) induce $\ZZ_\ell$-linear maps $\ZZ_\ell[X_{n+1}(x)]\to\ZZ_\ell[X_n(x)]$ via $[\xi]\mapsto[\varphi(\xi)]$.  The arboreal Tate module $T_\ell(x) = \varprojlim_n \ZZ_\ell[X_n(x)]$, formed from these maps, is an $\ell$-adic Galois representation. By Theorem~\ref{thm:main}, if $\varphi$ has strict good reduction and $x\in\OO_K$ has reduction $\bar x$ outside $\PC(\widetilde\varphi)$, then $T_\ell(x)$ is \emph{unramified} as a $G_K$-representation.

\item \text{Case $\ell = p$:} This is a $p$-adic representation. \textbf{(Corrected in v4.)} Calling it \emph{crystalline} requires constructing a compatible Frobenius structure on the tower. It is natural to ask: what additional data on the tower---beyond crystallinity of the representation at each finite level---would suffice to detect potential good reduction? Crystallinity alone is insufficient, as degenerate towers (e.g.\ a totally invariant fixed point with trivial, hence crystalline, representation) need not have potentially good reduction. For degree bounds on \emph{attaining} potentially good reduction after a finite extension (via conjugation), compare \cite{BenedettoIMRN2015}. Related $p$-adic NOS-type criteria are known for curves via the $p$-adic unipotent fundamental group \cite{AndreattaIovitaKim2015} and for $K3$ surfaces via monodromy/crystalline methods \cite{PerezBuendiaAMQ2019,HernandezMadaPadova2021,LiedtkeMatsumoto2018}. A complete treatment of the crystalline case requires constructing compatible Frobenius structures on the tower and is the natural sequel to the present paper; here we use the ramification consequences of Theorem~\ref{thm:main}.
\end{itemize}

\backmatter

\section*{Version history}
\noindent\textbf{v4 (arXiv).} This version incorporates the corrections contained in the accompanying corrigendum submitted to Research in Number Theory (Res.\ Number Theory \textbf{12} (2026), article~60, \href{https://doi.org/10.1007/s40993-026-00748-9}{doi:10.1007/s40993-026-00748-9}):
\begin{itemize}[nosep]
\item The Consequence of Theorem~3.1 now includes the hypothesis $\bar x\neq\tilde\varphi^n(\infty)$ for the affine chart formulation.
\item Proposition~7.1(c) uses the scheme-theoretic condition $\tilde\varphi^{-1}(U_k)\to U_k$ finite \'etale of degree~$1$.
\item The orbital colimit $T_{O^+(x)}$ is $(d+1)$-regular in the regular case; the profinite target is $\Aut_{\mathrm{spine}}(T_{O^+(x)})$.
\item Section~4.3 includes the uniformity caveat for the finite set~$S$ along infinite orbits.
\item Remark~\ref{rem:special-small} clarifies that it is the orbital Galois \emph{image} that is small for special maps, not the automorphism group of the tree.
\item The proof of Proposition~4.7 now references the projective fibre argument for the general case.
\item Remark~\ref{rem:negative-Q411} (new) provides two independent counterexamples to Question~\ref{q:orbital-regular}: a covering obstruction ($\varphi(z)=z^3+1$, $\operatorname{Deck}\neq 1$) and a postcritical obstruction ($f(z)=z^3-\tfrac{3}{2}z^2$, $\operatorname{Deck}=1$, PCF). A compositional obstruction ($\varphi=g^2+c$) is noted as a third source. The open-image question is reformulated.
\item Remark~\ref{rem:naive_fails} now includes the precise parity criterion for ramification of $K(X_1(x))/K$ over $K=\QQ_p$ when $\bar x\in\PC(\tilde\varphi)$; both bullets no longer assert unconditional ramification (the first bullet's false implication ``negative valuation $\Rightarrow$ ramified'' is corrected to the parity criterion $v_p(x-p)=v_p(x)$), and the corrected paragraph uses $K(\sqrt{x-p})$ (not $K(\sqrt{p-x})$) and handles the case $x=p$ explicitly.
\item The proof of Proposition~4.7 now makes explicit that the affine chart formulation requires $x_m\in\OO_K$.
\item Definition~\ref{def:residually-admissible} conditions the identity $\overline{\varphi^m(x)}=\tilde\varphi^m(\bar x)$ on strict good reduction; the codomain of the orbital reduction map is stated as a subset of $\PP^1(\bar k)$, with the equality $\widetilde{X}_n(\bar x_m)=\tilde\varphi^{-n}(\bar x_m)$ holding under strict good reduction.
\item Remark~\ref{rem:general-wreath} clarifies that $\Aut_{\mathrm{spine}}$ acts on the full preimage graph (vertices and edges) and provides a formal graph structure on $T_{O^+(x)}$.
\item The abstract distinguishes the projective fibre (no chart condition) from the affine fiber polynomial (chart condition $\bar x\neq\tilde\varphi^n(\infty)$).
\item The adelic outlook paragraph (Section~\ref{sec:orbital-open-image}) is conditioned on the existence of a uniform~$S$.
\item Remark~\ref{rem:critdivisor} notes that when $p\mid e$ the Wronskian may vanish identically; the intrinsic definition via $\Omega_{\widetilde\varphi}$ is recorded.
\item The arboreal Tate module $T_\ell(x)$ (\S\ref{subsec:tate}) now includes explicit transition maps $[\xi]\mapsto[\varphi(\xi)]$.
\item The case $d=1$ now correctly states ``defined over $K(x)$'' instead of ``defined over $K$''.
\item The crystalline passage is reformulated as an open question.
\end{itemize}
Corrected passages are marked \textbf{(Corrected in v4)} or \textbf{(Added in v4)} inline.

\section*{Declarations}
\noindent\textbf{Funding.} The author was partially supported by the SECIHTI (Ciencia de Frontera) grant CF 2019/217367.\\[2pt]
\noindent\textbf{Competing interests.} The author declares no competing interests.\\[2pt]
\noindent\textbf{Data availability.} No datasets were generated or analyzed during the current study.

\end{document}